\documentclass[12pt]{article}

\usepackage{float}
\usepackage{tikz}
\usepackage{diagbox}
\usepackage{capt-of}
\usepackage{color}
\usepackage{fullpage}
\usepackage{amssymb}
\usepackage{amsmath}
\usepackage{amsthm}
\usepackage{amsfonts}
\usepackage{amscd}
\usepackage{appendix}

\usepackage[colorlinks=true,
linkcolor=brown,
filecolor=brown,
citecolor=brown]{hyperref}

\newcommand{\seqnum}[1]{\href{https://oeis.org/#1}{\rm \underline{#1}}}
\DeclareMathOperator{\re}{re}

\begin{document}

\theoremstyle{definition}
\newtheorem{thm}{Theorem}
\newtheorem{cor}[thm]{Corollary}
\newtheorem{lem}[thm]{Lemma}
\newtheorem{proposition}[thm]{Proposition}
\newtheorem{example}[thm]{Example}

\begin{center}
\vskip 1cm{\Large\bf 
Divisibility properties of Dedekind numbers
}
\vskip 1cm
\large
Bart\l{}omiej Pawelski\\
Institute of Informatics\\
University of Gda\'nsk\\
Poland\\
\href{mailto:bartlomiej.pawelski@ug.edu.pl}{\tt bartlomiej.pawelski@ug.edu.pl} \\
\ \\
Andrzej Szepietowski\\
Institute of Informatics\\
University of Gda\'nsk\\
Poland\\
\href{mailto:andrzej.szepietowski@ug.edu.pl}{\tt andrzej.szepietowski@ug.edu.pl} \\
\end{center}

\vskip .2in

\begin{abstract}
We study some divisibility properties of Dedekind numbers. We show that the ninth Dedekind number is congruent to 6 modulo 210.
\end{abstract}

\section{Introduction}\label{sec:intro}
  
We define $D_n$ to be the set of all monotone Boolean functions of $n$ variables. The cardinality of this set---$d_n$ is known as the $n$--th Dedekind number. Values of $d_n$ are described by the OEIS \textit{(On-Line Encyclopedia of Integer Sequence)} sequence \seqnum{A000372} (see Table \ref{tab:RD}).

In 1990, Wiedemann calculated $d_8$ \cite{Wied1991}. His result was confirmed in 2001 by Fidytek, Mostowski, Somla, and Szepietowski \cite{Fid2001}. The impulse for writing our paper came from the letter from Wiedemann to Sloane \cite{Wiedlet}
 informing about the computation of the eighth Dedekind number, specifically this fragment: ``Unfortunately, I don't see how to test it...''. Wiedemann only knew that $d_8$ is even. Despite its obvious importance, there is a lack of studies on the divisibility of Dedekind numbers. As far as we know, the only paper concerning this question is Yamamoto's paper \cite{yamamoto}, where he shows that if $n$ is even, then $d_n$ is also even. 
 
 Our research aims to fill this lack by proposing new methods to determine the divisibility of Dedekind numbers. As an application of these methods, we compute remainders of $d_9$ divided by one-digit prime numbers, which (we hope) will help to verify the value $d_9$ after its first computation.
 
Our main result is the following system of congruences:

\[ d_9 \equiv 0\;(\bmod\;2), \]
\[d_9 \equiv 0\;(\bmod\;3), \]
\[d_9 \equiv 1\;(\bmod\;5), \]
\[d_9 \equiv 6\;(\bmod\;7). \]

By the Chinese remainder theorem, we have
\[ d_9 \equiv 6\;(\bmod\;210). \]

\section{Preliminaries}

Let $B$ denote the set $\{0,1\}$ and $B^n$ the
set of $n$-element sequences of $B$.
A Boolean function with $n$ variables is any function
from $B^n$ into $B$. There are $2^n$ elements in $B^n$
and $2^{2^n}$ Boolean functions with $n$ variables.
There is the order relation in $B$ (namely: $0\le 0$, $\;0\le1$, $\;1\le1$)
and the partial order in $B^n$:
for any two elements: 
$x=(x_1,\dots,x_n)$, $\;y=(y_1,\dots,y_n)$ in $B^n$,
$$x\le y\quad\hbox{if and only if}\quad x_i\le y_i
\quad \hbox{for all $1\le i\le n$}.$$
A function $h:B^n\to B$ is monotone if
$$x\le y \Rightarrow h(x)\le h(y).$$
Let $D_n$ denote the set of monotone functions with $n$ variables
and let $d_n$ denote $|D_n|$.
We have the partial order in $D_n$ defined by:
$$g\le h\quad\hbox{if and only if}\quad g(x)\le h(x)\quad
\hbox{for all } x\in B^n.$$
We shall represent the elements of $D_n$ as strings of bits of
length $2^n$. Two elements of $D_0$ will be represented as 0 and 1.
Any element $g\in D_1$ can be represented as 
the concatenation $g(0)*g(1)$, where $g(0), g(1)\in D_0$ and $g(0)\le g(1)$. Hence, $D_1=\{00, 01, 11\}$. 
Each element of $g\in D_2$ is the concatenation (string) of four bits: $g(00)*g(10)*g(01)*g(11)$ which can be represented as
a concatenation $g_0*g_1$, where $g_0, g_1\in D_1$ and $g_0\le g_1$.
Hence, $D_2=\{0000, 0001, 0011, 0101, 0111, 1111\}$.
Similarly, any element of $g\in D_n$ can be represented as
a concatenation $g_0*g_1$, where $g_0, g_1\in D_{n-1}$ and $g_0\le g_1$. 
Therefore, we can treat functions in $D_n$ as sequences of bits and as integers. By $\preceq$ we denote the total order in $D_n$ induced by the total order in integers. Additionally, we shall write 
$x\prec y$ if $x\preceq y$ and $x\ne y$.

For a set $Y\subseteq D_n$, by $Y^2$ we denote the Cartesian power $Y^2=Y\times Y$, that is the set of all ordered pairs $(x,y)$ with $x,y\in Y$. Similarly for more 
than two factors, we write $Y^k$ for the set of ordered $k$-tuples of elements of $Y$.
By $\top$ we denote the maximal element in $D_n$, that is $\top=(1\ldots 1)$; and by $\bot$ the minimal element in $D_n$, that is $\bot=(0\ldots 0)$. For two elements $x,y\in D_n$, by $x|y$ we denote the bitwise or; and by $x\&y$ we denote the bitwise and. Furthermore,
$\re(x,y)$ denotes $|\{z\in D_n : x\le z\le y\}|$. Note that $\re(x,\top)=|\{z\in D_n : x\le z\}|$ and $\re(\bot,y)=|\{z\in D_n : z\le y\}|$. 

\subsection{Posets}
A {\it poset (partially ordered set)} $(S,\le)$ consists of a set $S$ together with a binary relation (partial order) $\le$ which is reflexive, transitive, and antisymmetric.
 Given two posets $(S,\le)$ and $(T,\le)$ a function $f:S \to T$ is {\it monotone}, if  $x \le y$ implies $f(x)\le f(y)$.
By $T^S$ we denote the poset of all monotone functions from $S$ to $T$ with the partial order defined by:
$$f\le g\quad\hbox{ if and only if}\quad f(x)\le g(x)\hbox{ for all }x\in S.$$
In this paper we use the following well-known lemma; see \cite{szepietowski}:
\begin{lem}\label{L1} The poset $D_{n+k}$ is isomorphic to the poset $D_n^{B^k}$---the poset of monotone functions from $B^k$ to $D_n$.
\end{lem}

\subsection{Permutations and equivalence relation}
Let $S_n$ denote the set of permutations on $\{1,\ldots,n\}$.
Every permutation $\pi\in S_n$ defines the permutation on $B^n$ by
$\pi(x)=x\circ\pi$
(we treat each element $x\in B^n$ as a function
$x:\{1,\ldots,n\}\to \{0,1\}$).
Note that $x\le y$ if and only if $\pi(x)\le\pi(y)$.
The permutation $\pi$ also generates the permutation
on $D_n$. Namely, by $\pi(g)=g\circ \pi$.
Note that $\pi(g)$ is monotone if $g$ is monotone.
By $\sim$ we denote an equivalence relation on $D_n$. Namely,
two functions $f, g\in D_n$ are {\it equivalent}, $f\sim g$, if there is a 
permutation $\pi\in S_n$ such that $f=\pi(g)$.
For a function $f\in D_n$ its {\it equivalence class}
is the set $[f]=\{g\in D_n : g \sim f\}$. By $\gamma(f)$ we denote $|[f]|$.
For $m>1$, let $E_{n,m}= \{f\in D_n:\gamma(f)\equiv 0\pmod m\}$
and $E_{n,m}^c=D_n-E_{n,m}$.
For the class $[f]$, we define its {\it canonical representative} as the one element in $[f]$ chosen to represent the class. One of the possible approaches is to choose its minimal
(according to the total order $\preceq$) element \cite{Wied1991}. Sometimes we shall identify the class $[f]$ with its canonical representative and treat $[f]$ as an element in $D_n$.
By $R_n$ we denote the set of equivalence classes and by $r_n$ we denote the number of the equivalence classes; that is $r_n=|R_n|$. Values of $r_n$ are described by \seqnum{A003182} OEIS sequence;
see Table~\ref{tab:RD}.

\begin{table}[ht]
\centering
\begin{tabular}[h]{|l|l|l|}
\hline
 $n$ & $d_n$ & $r_n$\\
\hline
0 & 2 & 2 \\
1 & 3 & 3 \\
2 & 6 & 5 \\
3 & 20 & 10 \\
4 & 168 & 30 \\
5 & 7,581 & 210 \\
6 & 7,828,354 & 16,353 \\
7 & 2,414,682,040,998 & 490,013,148 \\
8 & 56,130,437,228,687,557,907,788 & 1,392,195,548,889,993,358 \\

\hline
\end{tabular}
\caption{Known values of $d_n$ (\seqnum{A000372}) and $r_n$ (\seqnum{A003182})}
\label{tab:RD}
\end{table}

\section{Divisibility of Dedekind numbers by 2}

 In 1952, Yamamoto \cite{yamamoto} proved that if $n$ is even, then $d_n$ is also even. Obviously, the parity of $d_9$ cannot be checked with this property. However, there are other methods to check it. One of the possible approaches is by using the duality property of Boolean functions. 
For each $x\in D_n$, we have {\it dual} $x^d$ which is obtained by reversing and negating
all bits. For example, $1111^d=0000$ and $0001^d=0111$. An element $x\in D_n$ is {\it self-dual} if $x=x^d$. For example, $0101$ and $0011$ are self-duals in $D_2$.
If $x$ is not self-dual, and $y=x^d\ne x$, then $y^d=x$. Thus, non-self-duals form pairs
 of the form $(x,x^d)$, where $x\ne x^d$. Let $k_n$ denote the number of these pairs and let $\lambda_n$ denote the number of self-dual functions in $D_n$.
We have that $d_n=2k_n + \lambda_n$.
Hence, 
$\lambda_n \equiv d_n\pmod2$.
Values of $\lambda_n$ are described by the OEIS sequence \seqnum{A001206};
see Table~\ref{tab:SD}. The last known term of this sequence---$\lambda_9$ was calculated in 2013 by Brouwer et al. \cite{brouwer}.

\begin{table}[ht]
\centering
\begin{tabular}[h]{|l|l|}
\hline
 $n$ & $\lambda_n$\\
\hline
0 & 0 \\
1 & 1 \\
2 & 2 \\
3 & 4 \\
4 & 12 \\
5 & 81 \\
6 & 2,646 \\
7 & 1,422,564 \\
8 & 229,809,982,112 \\
9 & 423,295,099,074,735,261,880 \\

\hline
\end{tabular}
\caption{Known values of $\lambda_n$ (\seqnum{A001206})}
\label{tab:SD}
\end{table}

\begin{cor}\label{corollary1}
$d_9 \equiv \lambda_9 \equiv 0 \pmod2 $
\end{cor}

One can directly check that $d_n \equiv \lambda_n\pmod 2$ for $n \le 8$.

\section{Divisibility of Dedekind numbers by 3}\label{sec:Sec4}

By Lemma~\ref{L1}, the poset $D_{n+3}$ is isomorphic to the poset $D_n^{B^3}$---the set of monotone functions from 
$B^3=\{000,001,010,100,110,101,011,111\}$ to $D_n$. 
\begin{def}\label{def1}
Consider the function $H$ which for every triple $(x,y,z)\in D_n^3$ returns the value
$$H(x,y,z)=\re(\bot,x\&y\&z)\cdot\sum_{s\ge x|y|z} \re(x|y,s)\cdot \re(x|z,s)\cdot \re(y|z,s).$$
\end{def}
Observe that $H(x,y,z)$ is equal to the number of monotone functions $f\in D_n^{B^3}$ with
$f(001)=x$, $f(010)=y$ and $f(100)=z$. Indeed, for $f(000)$ we can choose any value $t$ 
satisfying $t\le x\& y\&z$, which can be done in $\re(\bot,x\& y\&z)$ ways. For $f(111)$ we can choose any element in $\{s\in D_n: s\ge x|y|z\}$. Furthermore, after the value $f(111)=s$ is chosen, then the values
$f(011)$, $f(101)$, $f(110)$ can be chosen independently
from each other. We can choose a value for $f(011)$ in $\re(x|y,s)$ ways; a value for $f(101)$ in $\re(x|z,s)$ ways; and a value for $f(110)$ in $\re(y|z,s)$ ways.

\def\drawconnectionsa{\draw (0,0) -- (-2,2);
  \draw (0,0) -- (0,2);
  \draw (0,0) -- (2,2);
  \draw (-2,2) -- (-2,4);
  \draw (-2,2) -- (0,4);
  \draw (0,2) -- (-2,4);
  \draw (0,2) -- (2,4);
  \draw (2,2) -- (0,4);
  \draw (2,2) -- (2,4);
  \draw (-2,4) -- (0,6);
  \draw (0,4) -- (0,6);
  \draw (2,4) -- (0,6);}
  
\def\drawconnectionsb{
  \draw (0,0) -- (0,2);
  \draw (0,2) -- (-2,4);
  \draw (0,2) -- (0,4);
  \draw (0,2) -- (2,4);
  \draw (-2,4) -- (0,6);
  \draw (0,4) -- (0,6);
  \draw (2,4) -- (0,6);}
  
\def\drawconnectionsc{
  \draw (0,0) -- (0,2);
  \draw (0,2) -- (0,4);
  \draw (0,4) -- (0,6);}

\begin{center}
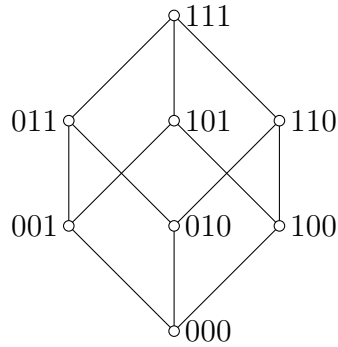

\begin{tikzpicture}[scale=0.7]
  \drawconnectionsa{}
  \draw[black, fill=white] (0,0) circle (.1cm) node[right] {$000$};
  \draw[black, fill=white] (-2,2) circle (.1cm) node[left] {$001$};
  \draw[black, fill=white] (0,2) circle (.1cm) node[right] {$010$};
  \draw[black, fill=white] (2,2) circle (.1cm) node[right] {$100$};
  \draw[black, fill=white] (-2,4) circle (.1cm) node[left] {$011$};
  \draw[black, fill=white] (0,4) circle (.1cm) node[right] {$101$};
  \draw[black, fill=white] (2,4) circle (.1cm) node[right] {$110$};
  \draw[black, fill=white] (0,6) circle (.1cm) node[right] {$111$};
  \end{tikzpicture}
  \captionof{figure}{Poset $B^3$.}
  \label{fig:1P3}
\end{center}

For $A\subseteq D_n^3=D_n\times D_n\times D_n$, let $H(A)$ denote $\sum_{(x,y,z)\in A}H(x,y,z)$.
By Lemma~\ref{L1}, we have that
 $$d_{n+3}=H(D_n^3)=\sum_{x\in D_n}\sum_{y\in D_n}
\sum_{z\in D_n}H(x,y,z).$$

The value of $H$ is invariant under any permutation of its arguments. Let $x,y,z\in D_n$ and suppose that
$x\preceq y\preceq z$. We have the following types of equivalence classes:

\begin{enumerate}
  \item $x=y=z$, with one element in class
  \item $x\prec y=z$, with three elements in class
  \item $x=y\prec z$, with three elements in class
  \item $x\prec y\prec z$, with six elements in class.
\end{enumerate}

A similar property is exploited by Fidytek et al. \cite[Algorithm 3]{Fid2001} to speed up their computation of $d_8$. However, they work with a more complex level-4-algorithm. After applying it to our situation, $d_{n+3}$ can be expressed as the following summation:

\begin{equation} \label{d3-2}
d_{n+3} = \sum_{x\in D_n} H(x,x,x) 
+3 \cdot \sum_{\substack{x,y\in D_n \\ x\prec y}} (H(x,x,y) + H(x,y,y))
+ 6 \cdot \sum_{\substack{x,y,z\in D_n \\ x\prec y\prec z}} H(x,y,z)
\end{equation}

Thus,
$$d_{n+3} \equiv \sum_{x\in D_n} H(x,x,x)\pmod 3.$$
Now we can use one more iteration of symmetry exploitation.

\begin{center}
\begin{tikzpicture}[scale=0.7]
  \drawconnectionsb{}
  \draw[black, fill=white] (0,0) circle (.1cm) node[right] {$a$};
  \draw[black, fill=white] (0,2) circle (.1cm) node[right] {$b$};
  \draw[black, fill=white] (-2,4) circle (.1cm) node[left] {$c$};
  \draw[black, fill=white] (0,4) circle (.1cm) node[right] {$d$};
  \draw[black, fill=white] (2,4) circle (.1cm) node[right] {$e$};
  \draw[black, fill=white] (0,6) circle (.1cm) node[right] {$f$};
  \end{tikzpicture}
  \captionof{figure}{Poset $C$.}
  \label{fig:2P3}
\end{center}
Consider the poset $C$ presented in Figure~\ref{fig:2P3}.
Note that $\sum_{x\in D_n} H(x,x,x)=|D_n^C|$---the number of functions from the poset $C$ to $D_n$.
Let $H'$ denote the function which for each $u,v,w\in D_n$ returns
\begin{equation} \label{H'}
H'(u,v,w) = \Big( \sum_{s \le(u\& v\& w)} re[s, u\& v\& w] \Big) \cdot re[u|v|w, \top].
\end{equation}
Observe that $H'(u,v,w)$ is equal to the number of monotone functions $h\in D_n^{C}$ with
$h(c)=u$, $h(d)=v$ and $h(e)=w$. 
Indeed, for $h(f)$ we can choose any value $t$ 
satisfying $t\ge u|v|w$, which can be done in $\re(u|v|w, \top)$ ways. For $h(a)$ we can choose any element in $\{s\in D_n: s\le u\& v\& w\}$. Furthermore, after the value $h(a)=s$ is chosen, then the value $h(b)$ can be chosen in $\re(s,u\& v\& w)$ ways.
Hence,
\begin{equation} \label{c3-1}
\sum_{x\in D_n} H(x,x,x) = \sum_{u,v,w\in D_n} H'(u,v,w).
\end{equation}
The value of $H'$ is invariant under any permutation of its arguments, so we can use the same symmetry--based simplification as in Equation \ref{d3-2}:

\begin{multline}
\label{c3-2}
 \sum_{u,v,w\in D_n} H'(u,v,w)= \sum_{u\in D_n} H'(u,u,u) \\ + 3\cdot\sum_{\substack{u,v\in D_n \\ u\prec v}} (H'(u,u,v) + H'(u,v,v)) +6\cdot \sum_{\substack{u,v,w \in D_n \\ u\prec v\prec w}} H'(u,v,w) 
\end{multline}
Thus, we have 
$$d_{n+3}\equiv \sum_{u\in D_n} H'(u,u,u)\pmod3$$
Observe that $\sum_{u\in D_n} H'(u,u,u)$ is equal to the number of monotone functions from
the path $P_4$ to $D_n$.
Szepietowski~\cite{szepietowski} shows that monotone functions from the path $P_4=(a<b<c<d)$ to a poset $S$, are connected to the elements of the array 
$M(S)^3=M(S)\times M(S)\times M(S)$, where $M(S)$ is the array of the poset $S$. Namely, for $i,j\in S$, we have $M(S)[i,j]=1$ if $i\le j$, and $M(S)[i,j]=0$ otherwise. Moreover the sum of elements of $M(S)^3$ is equal to $|S^{P_4}|$.
For example, for the poset
$D_1=\{00<01<11\}$, we have
$$
M(D_1)=\left(
\begin{array}{ccc}
1 & 1& 1\\
0 & 1& 1\\
0 & 0&1
\end{array}
\right)
$$
and
$$
M(D_1)^3=\left(
\begin{array}{ccc}
1 & 3& 6\\
0 & 1& 3\\
0 & 0&1
\end{array}
\right)
$$
The sum of the elements of $(M(D_1)^3)$ is equal to 15, which is equal to $|D_1^{P_4}|$---the number of monotone functions from $P_4$ to $D_1$. 

Furthermore, consider the poset $D_2=\{0000,0001,0011,0101,0111,1111\}$ and its array
$$
M(D_2)=\left(
\begin{array}{cc|cc|cc}
1& 1& 1& 1& 1& 1 \\
0& 1& 1& 1& 1& 1 \\
\hline
0& 0& 1& 0& 1& 1 \\
0& 0& 0& 1& 1& 1 \\
\hline
0& 0& 0& 0& 1& 1 \\
0& 0& 0& 0& 0& 1
\end{array}
\right)
$$
Consider now the array 

$$
M(D_2)^3=\left(
\begin{array}{cc|cc|cc}
1& 3& 6& 6& 14& 20 \\
0& 1& 3& 3& 9& 14 \\
\hline
0& 0& 1& 0& 3& 6 \\
0& 0& 0& 1& 3& 6 \\
\hline
0& 0& 0& 0& 1& 3 \\
0& 0& 0& 0& 0& 1
\end{array}
\right)
$$
The sum of the elements of $(M(D_2)^3)$ is equal to 105, which is equal to $|D_2^{P_4}|$---the number of monotone functions from $P_4$ to $D_2$.
In a similar we can compute $|D_n^{P_4}|$ for $n=3, 4, 5 $. Unfortunately, this method cannot be easily applied for $n=6$, because the array $M(D_6)$ is too big. However, Pawelski~\cite{pawelski} proposes another method: $|D_{(n+m)}^{P_4}| = |D_n^{ P_4\times B^m}| = |(D_n^{ P_4})^{B^m}|$ (also see \cite{szepietowski}).
Using the same program as used to compute $|D_5^{P_4}|$ \cite{pawelski} we can calculate $|D_6^{P_4}|$ and the result (see Table \ref{tab:FP3}) is divisible by 3.

\begin{cor}\label{corollary3-1}
As $|D_6^{P_4}| = 868329572680304346696$ is divisible by 3, then also $d_9$ is divisible by 3.
\end{cor}

\begin{table}[ht]
\centering
\begin{tabular}[h]{|l|l|c|}
\hline
 $n$ & $|D_n^{P_4}|$ & $|D_n^{P_4}| \bmod 3$\\
\hline
0 & 5 & 2 \\
1 & 15 & 0 \\
2 & 105 & 0\\
3 & 3,490 & 1\\
4 & 2,068,224 & 0 \\
5 & 262,808,891,710 & 1\\
6 & 868,329,572,680,304,346,696 & 0\\

\hline
\end{tabular}
\caption{Known values of $|D_n^{P_4}|$. Note that $d_{n+3}\equiv |D_n^{P_4}| \pmod3$.}
\label{tab:FP3}
\end{table}

One can directly check that $d_{n+3} \equiv |D_n^{P_4}|\pmod 3$ for $n \le 5$.

\section{Counting functions from $B^2$ to $D_n$}\label{sec:b2sec}
By Lemma~\ref{L1}, the poset $D_{n+2}$ is isomorphic to the poset $D_n^{B^2}$---the poset of monotone functions from $B^2=\{00,01,10,11\}$ to $D_n$.
Consider the function $G$ which for every pair $(x,y)\in D_n^2$ returns the value
$G(x,y)=\re(x|y,\top)\cdot \re(\bot,x\&y)$.
Observe that $G(x,y)$ is equal to the number of functions $f\in D_n^{B^2}$ with
$f(01)=x$ and $f(10)=y$. Indeed, we can choose $f(11)$ in $\re(x|y,\top)$ ways
and $f(00)$ in $\re(\bot,x\&y)$ ways. Note that $\re(\bot,x\&y)=\re(x^d|y^d,\top)$, where $x^d$ is the dual of $x$.

For $A\subseteq D_n\times D_n$ let $G(A)$ denote $\sum_{(x,y)\in A}G(x,y)$.
By Lemma~\ref{L1}, we have that
 $$d_{n+2}=G(D_n\times D_n)=\sum_{x\in D_n}\sum_{y\in D_n}G(x,y).$$
Observe that, for every permutation $\pi\in S_n$ and every $x, y\in D_n$, we have
$G(x,y)=G(\pi(x),\pi(y))$.
Indeed, $z\ge x|y$ iff $\pi(z)\ge \pi(x)|\pi(y)$.
Hence, $\re(x|y, \top)=\re(\pi(x)|\pi(y), \top)$. 
\begin{lem}\label{L2}
Let $Y$ be a subset $Y\subseteq D_n$ and suppose that $\pi(Y)=Y$ for every permutation $\pi\in S_n$;
and let $x$ and $y$ be two equivalent, $x\sim y$, elements in $D_n$. Then:
\begin{enumerate}
\item $G(\{x\}\times Y)=G(\{y\}\times Y)$.
\item $G([x]\times Y)=\gamma(x)\cdot G(\{x\}\times Y)$.
\item if $m$ divides $\gamma(x)$, then $m$ divides $G([x]\times Y)$
\item $m$ divides $G(E_{n,m}\times Y)$ and $G(Y\times E_{n,m})$
\end{enumerate}
\end{lem}

\begin{proof}
Notice that condition $\pi(Y)=Y$ implies that $\pi$ is a bijection on $Y$, or in other words, $\pi$ permutes the elements of $Y$.

(1) $G(\{x\}\times Y)=\sum_{s\in Y}G(x,s)=\sum_{s\in Y}G(\pi(x),\pi(s)) = \sum_{t\in \pi(Y)}G(\pi(x),t)= \\ = \sum_{t\in Y}G(\pi(x),t)=G(\{\pi(x)\}\times Y)$. 
We use the fact that $\pi(Y)=Y$.

(4) Observe that $E_{n,m}\times Y$ is a sum of disjoint sets of the form $[x]\times Y$, where $\gamma(x)$ is divisible by $m$.

\end{proof}

\begin{lem}\label{lemn2}
For each $Y=D_n$, $[y]$, $E_{n,m}$, or $E_{n,m}^c$, and each permutation $\pi\in S_n$, we have 
$\pi(Y)=Y$.
\end{lem}
As an immediate corollary, we have that
$$d_{n+2}=
\sum_{x\in R_n}\gamma(x)\cdot G(\{x\}\times D_n).$$

\begin{example}\label{exd2}
Consider the poset $D_2=\{0000,0001,0011,0101,0111,1111\}$. There are five equivalent classes: namely, $R_2=\{ \{0000\},\{0001\},\{0011, 0101\},\{0111\},\{1111\}\}$.
Two elements: $0101$ and $0011$ are equivalent.
For $m=2$, we have 
$E_{2,2}=\{0011,0101\}$
and 
$E^c_{2,2}=\{0000,0001,0111,1111\}.$
Table \ref{tab:xy1} presents values of $G(x,y)$ for $x,y\in D_2$.
Let $Y=[0011]=\{0011,0101\}$.  For every permutation $\pi\in S_2$, we have $\pi(Y)=Y$.
Furthermore, $G(\{0011\}\times Y)=G(\{0101\}\times Y)=9+4=13$ ; and $G([0011]\times Y) = 2 \cdot 13 = 26$, which is divisible by 2.

Similarly, for $Z=[0001]=\{0001\}$, we have that $\pi(Z)=Z$ for every permutation $\pi\in S_2$.
Furthermore, $G(\{0011\}\times Z)=G(\{0101\}\times Z)=6$ ; and $G([0011]\times Z) = 2 \cdot 6 = 12$, which is divisible by 2. By summing up all values in Table~\ref{tab:xy1} we obtain $G(D_2\times D_2)=168=d_4$.
\end{example}

\begin{table}[ht]
\centering
\begin{tabular}[h]{|c|c|c|c c|c|c|}
\hline
 \diagbox{$x$}{$y$} & 0000 & 0001 & 0011 & 0101 & 0111 & 1111 \\
\hline
0000 & 6 & 5 & 3 & 3 & 2 & 1 \\
\hline
0001 & 5 & 10 & 6 & 6 & 4 & 2 \\
\hline
0011 & 3 & 6 & 9 & 4 & 6 & 3\\
0101 & 3 & 6 & 4 & 9 & 6 & 3\\
\hline
0111 & 2 & 4 & 6 & 6 & 10 & 5 \\
\hline
1111 & 1 & 2 & 3 & 3 & 5 & 6\\

\hline
\end{tabular}
\caption{Values of $G(x,y)$ for $x,y\in D_2$.}
\label{tab:xy1}
\end{table}

\begin{thm}\label{T3}
$$d_{n+2}\equiv G(D_n\times D_n)\equiv 
G(E_{n,m}^c\times E_{n,m}^c)\pmod m$$
and 
$$ d_{n+2}\equiv 
\sum_{x\in R_n\cap E_{n,m}^c}\sum_{y\in E_{n,m}^c} \gamma(x)\cdot G(x,y)\pmod m.$$
Here we identify each class $[x]\in R_n$ with 
its canonical representative.
\end{thm}

\begin{proof}
Let $A_1$ denote $E_{n,m}\times D_n$ 
and $A_2$ denote $D_n\times E_{n,m}$.
Observe that
$$A_1\cap A_2= E_{n,m}\times E_{n,m}$$
$$E_{n,m}^c\times E_{n,m}^c=D_n\times D_n- (A_1\cup A_2)$$
and
$$G(E_{n,m}^c\times E_{n,m}^c)=G(D_n\times D_n)- G(A_1)
-G(A_2)+G(A_1\cap A_2).$$
By Lemma~\ref{L2}(2), we have  that $m$ divides 
$G(A_1)$, $G(A_2)$, and $G(A_1\cap A_2)$.
Hence,
$$d_{n+2}\equiv G(D_n\times D_n)\equiv 
G(E_{n,m}^c\times E_{n,m}^c)\pmod m.$$
Observe that $\pi(E_{n,m}^c)=E_{n,m}^c$, for every permutation $\pi\in S_n$. Hence, by Lemma~\ref{L2}(4):
$$ d_{n+2}\equiv 
\sum_{x\in R_n\cap E_{n,m}^c}\sum_{y\in E_{n,m}^c} \gamma(x)\cdot G(x,y)\pmod m.$$
\end{proof}

\begin{example}[Continuation of Example \ref{exd2}]
  By summing the relevant values listed in Table \ref{tab:xy1}, we obtain 
  $G(E_{2,2}^c\times E_{2,2}^c)=6+5+2+1+5+10+4+2+2+4+10+5+1+2+5+6=70$. By Theorem \ref{T3}, we have $d_4 \equiv 70 \pmod 2$. Indeed, $d_4 = 168$, which is even.
\end{example}

\section{Counting functions from $B^3$ to $D_n$}
In the next two sections, we show that similar techniques can be also applied to functions
in $D_n^{B^3}$ and $D_n^{B^4}$.
In Section \ref{sec:Sec4} we define the function $H$. We have shown that
$$d_{n+3}=H(D_n^3)=\sum_{x\in D_n}\sum_{y\in D_n}
\sum_{z\in D_n}H(x,y,z)$$
Observe that for every permutation $\pi\in S_n$ and every $x, y,z\in D_n$, we have
$H(x,y,z)=H(\pi(x),\pi(y),\pi(z))$.
\begin{lem}\label{L4}
Let $Y$ and $Z$ be subsets $Y,Z\subseteq D_n$ and suppose that $\pi(Y)=Y$ and $\pi(Z)=Z$ for every permutation $\pi\in S_n$; and let $x$ and $y$ be two equivalent, $x\sim y$, elements in $D_n$. Then:
\begin{enumerate}
\item $H(\{x\}\times Y\times Z)=H(\{y\}\times Y\times Z)$.
\item $H([x]\times Y\times Z)=\gamma(x)\cdot H(\{x\}\times Y\times Z)$.
\item if $m$ divides $\gamma(x)$, then $m$ divides $H([x]\times Y\times Z)$
\item $m$ divides $H(E_{n,m}\times Y\times Z)$.
\end{enumerate}
\end{lem}
\begin{proof}
(1) $H(\{x\}\times Y\times Z)=\sum_{s\in Y}\sum_{t\in Z}H(x,s,t)=\sum_{s\in Y}\sum_{t\in Z}H(\pi(x),\pi(s),\pi(t))=
\sum_{u\in \pi(Y)}\sum_{v\in \pi(Z)}H(\pi(x),u,v)=\sum_{u\in Y}\sum_{\in  Z}H(\pi(x),u,v)=H(\{\pi(x)\}\times Y\times Z)$. 
We use the fact that $\pi$ is a bijection on $Y\times Z$ and permutes the elements of $Y\times Z$.
\end{proof}

As an immediate corollary, we have 
$$d _{n+3}= H(D_n\times D_n\times D_n)=\sum_{x\in R_n}\gamma(x)\cdot H(\{x\}\times D_n\times D_n).$$
\begin{thm}\label{T5}
$$d_{n+3}\equiv H(D_n\times D_n\times D_n)\equiv 
H(E_{n,m}^c\times E_{n,m}^c\times E_{n,m}^c)\pmod m$$
and
$$ d_{n+3}\equiv 
\sum_{x\in R_n\cap E_{n,m}^c}\sum_{y\in E_{n,m}^c}\sum_{z\in E_{n,m}^c} \gamma(x)\cdot
H(x,y,z)\pmod m.$$
Here, again, we identify each class $[x]\in R_n$ with 
its canonical representative.
\end{thm}

\begin{proof}
Let $A_1, A_2, A_3$ denote 
$E_{n,m}\times D_n\times D_n$,
$D_n\times E_{n,m}\times D_n$, and
$D_n\times\ D_n\times E_{n,m}$, respectively. 
Observe that
$$A_1\cap A_2= E_{n,m}\times E_{n,m}\times D_n$$
$$A_1\cap A_3= E_{n,m}\times D_n\times E_{n,m}$$
$$A_2\cap A_3= D_n\times E_{n,m}\times E_{n,m}$$
$$A_1\cap A_2\cap A_3= E_{n,m}\times E_{n,m}\times E_{n,m}$$
$$E_{n,m}^c\times E_{n,m}^c\times E_{n,m}^c=
D_n\times D_n\times D_n- \bigcup_{i=1}^3 A_i$$
Similarly as in \cite[Chapter 2.]{martin}, one can prove the following form of the principle of inclusion and exclusion:
$$H(E_{n,m}^c\times E_{n,m}^c\times E_{n,m}^c)=
H(D_n\times D_n\times D_n- \bigcup_{i=1}^3 A_i)$$
$$=H(D_n\times D_n\times D_n)- H(A_1)-H(A_2)-H(A_3)$$
$$ + H(A_1\cap A_2)+ H(A_1\cap A_3)+
H(A_2\cap A_3)-H(A_1\cap A_2\cap A_3).$$
Moreover, we have  that $m$ divides the following numbers:
\begin{itemize}
\item $H(A_i)$ for $i=1,2,3$
\item $H(A_i\cap A_j)$ for $1\le i< j\le 3$
\item $H(A_1\cap A_2\cap A_3)$.
\end{itemize} 
Hence, 
$$d _{n+3}\equiv H(D_n\times D_n\times D_n)\equiv H(E^c_{n,m}\times E^c_{n,m}\times E_{n,m}^c)\pmod m.$$
Observe that $\pi(E_{n,m}^c)=E_{n,m}^c$, for every $\pi\in S_n$, hence by Lemma~\ref{L4}(2), we have 
$$ d_{n+3}\equiv 
\sum_{x\in R_n\cap E_{n,m}^c}\sum_{y\in E_{n,m}^c}\sum_{z\in E_{n,m}^c} \gamma(x)\cdot
H(x,y,z)\pmod m.$$
\end{proof}

\begin{example}
Consider $D_4$. There are 168 elements in $D_4$ and 30 equivalent classes in $R_4$. The distribution of these equivalence classes based on their $\gamma$ value is presented in Table \ref{tab:r4gamma}.
For instance, there are six equivalent classes $[x]$ with $\gamma(x)=1$, two equivalent classes with $\gamma(x)=3$, and so forth. For $m=2$, the set $E^c_{4,2}$ contains only twelve elements
and $R_4\cap E^c_{4,2}$ contains eight elements. Similarly, for $m=3$, the set $E^c_{4,3}$ contains 42 elements and $R_4\cap E^c_{4,3}$ consists of 15 elements.
\end{example}

\begin{example}
We have utilized a Java implementation of the Theorem \ref{T5}. For $n=4$ and $m=2,3,4,6,12$ we have:

$ d_7 \equiv 2320978352 \pmod 2$, therefore $d_7 \bmod 2 = 0$,

$ d_7 \equiv 74128573428 \pmod 3$, therefore $d_7 \bmod 3 = 0$,

$ d_7 \equiv 128268820802 \pmod 4$, therefore $d_7 \bmod 4 = 2$,

$ d_7 \equiv 89637133284 \pmod 6$, therefore $d_7 \bmod 6 = 0$,

$ d_7 \equiv 566167187562 \pmod {12}$, therefore $d_7 \bmod {12} = 6$,

One can check these values directly by dividing $d_7$ by 2, 3, 4, 6, and 12.
\end{example}

\begin{table}[ht]
\centering
\begin{tabular}[h]{| c | c |}
\hline
 $k$ & $|\{f\in R_4:\gamma(f) = k\}|$ \\
\hline
1 & 6 \\
3 & 2 \\
4 & 9 \\
6 & 6 \\
12 & 7 \\
\hline
\end{tabular}
\caption{Number of $f \in R_4$ with $\gamma(f) = k$.}
\label{tab:r4gamma}
\end{table}

\section{Counting functions from $B^4$ into $D_n$}
By Lemma~\ref{L1}, the poset $D_{n+4}$ is isomorphic to the poset $D_n^{B^4}$---the set of monotone functions from $B^4$
to $D_n$.
Consider the function $F$, which for every six elements $a,b,c,d,e,f\in D_{n}$, gets
how many functions $g\in D_n^{B^4}$ satisfy:
$g(0011)=a$,
$g(0101)=b$,
$g(1001)=c$,
$g(0110)=d$,
$g(1010)=e$,
$g(1100)=f$.
Or in other words $F$ counts in how many ways we can choose
the values of $g$ for other elements of $B^4$.
Observe that the values for upper elements:
$(1111)$, $(0111)$, $(1011)$, $(1101)$, $(1110)$ can be
chosen independently from the values for lower elements:
$(0000)$, $(0001)$, $(0010)$, $(0100)$, $(1000)$.

Consider now how many ways one can choose the values
for the upper elements:
$(1111)$, $(0111)$, $(1011)$, $(1101)$, $(1110)$.
First, we choose $g(1111)$ witch must be greater or equal than each of
the values $a,b,c,d,e,f$, so $g(1111)$ can be chosen
from the values greater or equal than $a|b|c|d|e|f$.
Next, if the value of $g(1111)$ is chosen, then the values
$g(0111)$, $g(1011)$, $g(1101)$, $g(1110)$ can be chosen independently
from each other. And
$g(0111)$ must be greater or equal than each one of
$g(0011)=a$,
$g(0101)=b$,
$g(0110)=d$. Hence $a|b|d\le g(0111)\le g(1111)$.
Thus $g(0111)$ can be chosen in $\re(a|b|d, g(1111))$ ways
(remember that $\re(u,v)$ is the number of elements between $u$ and $v$).
Similarly, the value of
$g(1011)$ can be chosen in $\re(a|c|e, g(1111))$ ways;
$g(1101)$ in $\re(b|c|f, g(1111))$ ways; and
$g(1110)$ in $\re(d|e|f, g(1111)]$ ways.

The number of ways $g$ can be extended to the upper elements is equal
to
$$H=\sum_{u\ge a|b|c|d|e|f}
 \re(a|b|d, u)\cdot \re(a|c|e, u)\cdot \re(b|c|f, u)\cdot \re(d|e|f, u).$$
Similarly one can show that
the number of ways $g$ can be extended to the lower elements is equal
to
$$h=\sum_{u\le a\& b\& c\& d\& e\& f}
\re(u,a\& b\& c)\cdot \re(u,a\& d\& e)\cdot \re(u,b\& d\& f)
\cdot \re(u,c\& e\& f).$$
Altogether there are $F(a,b,c,d,e,f)=H\cdot h$ functions satisfying:
$g(0011)=a$,
$g(0101)=b$,
$g(1001)=c$,
$g(0110)=d$,
$g(1010)=e$,
$g(1100)=f$.

For $A\subseteq (D_n)^6$ let $F(A)$ denote $\sum_{(a,b,c,d,e,f)\in A}F(a,b,c,d,e,f)$.
By Lemma~\ref{L1},
 $$d_{n+4}=F(D_n^6)=
\sum_{a\in D_n}\sum_{b\in D_n}\sum_{c\in D_n}
\sum_{d\in D_n}\sum_{e\in D_n}\sum_{f\in D_n}
F(a,b,c,d,e,f)$$

Observe that for every permutation $\pi\in S_n$ and every $a,b,c,d,e,f\in D_n$, we have
$F(a,b,c,d,e,f)=F(\pi(a),\pi(b),\pi(c),\pi(d),\pi(e),\pi(f))$.
Consider Cartesian product $Y=Y_1\times Y_2\times Y_3\times Y_4\times Y_5$ 
and let $\pi(y_1,\ldots,y_5)=(\pi(y_1),\ldots, \pi(y_5))$. Observe that, if $\pi(Y_i)=Y_i$ for every $i$, then $\pi(Y)=Y$ and $\pi$ permutes the elements of $Y$.

\begin{lem}\label{L6}
Let $Y$ be a subset $Y\subseteq D_n^5$ and suppose that $\pi(Y)=Y$ for every permutation $\pi\in S_n$; and let $x$ and $y$ be two equivalent, $x\sim y$, elements in $D_n$. Then:
\begin{enumerate}
\item $F(\{x\}\times Y)=F(\{y\}\times Y)$.
\item $F([x]\times Y)=\gamma(x)\cdot F(\{x\}\times Y)$.
\item If $m$ divides $\gamma(x)$, then $m$ divides $F([x]\times Y)$.
\item $m$ divides $F(E_{n,m}\times Y)$.
\end{enumerate}
\end{lem}

\begin{proof}
(1) $F(\{x\}\times Y)=\sum_{s\in Y}F(x,s)=\sum_{s\in Y}F(\pi(x),\pi(s))=
\sum_{u\in \pi(Y)}F(\pi(x),u)=\sum_{u\in Y}F(\pi(x),u)=F(\{\pi(x)\}\times Y)$. 
We use the fact that $\pi$ is a bijection on $Y$ and permutes the elements of $Y$.
\end{proof}

As a corollary we have
$$d _{n+4}= F(D_n^6)=F(D_n\times D_n\times D_n\times D_n\times D_n\times D_n)$$
and
$$d_{n+4}=
\sum_{x\in R_n}\gamma(x)\cdot F(\{x\}\times D_n\times D_n\times D_n \times D_n\times D_n).$$
\begin{thm}\label{T7}
$$d_{n+4}\equiv F(D_n^6)\equiv 
F((E_{n,m}^c)^6)\pmod m$$
and
$$ d_{n+4}\equiv 
\sum_{a\in R_n\cap E_{n,m}^c}\sum_{b\in E_{n,m}^c}\sum_{c\in E_{n,m}^c}
\sum_{d\in E_{n,m}^c}\sum_{e\in E_{n,m}^c}\sum_{f\in E_{n,m}^c}
 \gamma(a)\cdot F(a,b,c,d,e,f)\pmod m.$$
\end{thm}
\begin{proof}
Let $A_1$ denote $E_{n,m}\times D_n^5$;
let $A_2$ denote $D_n\times E_{n,m}\times D_n^4$; and so on.
More precisely, for $i$, $1\le i\le 6$, let $A_i$ denote $D_n^{i-1}\times E_{n,m}\times D_n^{6-i}$.
For any nonempty subset $I\subseteq\{1,\ldots,6\}$, let $A_I=\bigcap_{i\in I}A_i$.
Additionally, let $A_\emptyset=D_n^6$.
Observe that
$$(E_{n,m}^c)^6=
D_n^6- \bigcup_{i=1}^6 A_i.$$
Similarly as in \cite[Chapter 2.]{martin}, one can prove the following form of the principle of inclusion and exclusion:
$$F((E_{n,m}^c)^6)=
F(D_n^6- \bigcup_{i=1}^6 A_i)=\sum_{I\subseteq\{1,\ldots,6\}}(-1)^{|I|}F(A_I).$$
Moreover, if $I\ne\emptyset$, then $m$ divides $F(A_I)$.
Hence, 
$$d_{n+4}\equiv F(D_n^6)\equiv F((E^c_{n,m})^6)\pmod m.$$
Observe that $\pi(E_{n,m}^c)=E_{n,m}^c$, for every $\pi\in S_n$, hence by Lemma~\ref{L6}(2), we have 
$$ d_{n+4}\equiv 
\sum_{a\in R_n\cap E_{n,m}^c}\sum_{b\in E_{n,m}^c}\sum_{c\in E_{n,m}^c}
\sum_{d\in E_{n,m}^c}\sum_{e\in E_{n,m}^c}\sum_{f\in E_{n,m}^c}
 \gamma(a)\cdot F(a,b,c,d,e,f)\pmod m.$$
\end{proof}

 \begin{example}
We utilized a Java implementation of the Theorem \ref{T7}. For $n=4$ and $m=2,3,4,6,12$ we have:

$ d_8 \equiv 53336702474849828 
$, therefore $d_8 \bmod 2 = 0$,

$ d_8 \equiv 3019662424037271148 \pmod 3$, therefore $d_8 \bmod 3 = 1$,

$ d_8 \equiv 25754060568741983624 \pmod 4$, therefore $d_8 \bmod 4 = 0$,

$ d_8 \equiv 14729824485525634108 \pmod 6$, therefore $d_8 \bmod 6 = 4$,

$ d_8 \equiv 15054599294580333880 \pmod {12}$, therefore $d_8 \bmod {12} = 4$,

One can check these values directly by dividing $d_8$ by 2, 3, 4, 6, and 12.

\end{example}

\section{Application}\label{application}
To compute remainders of $d_9$ divided by 5 and 7, we chose the algorithm described in Section \ref{sec:b2sec}.
Our implementation lists all 490,013,148 elements of $R_7$ and calculates $\gamma(x)$ and $\re(\bot,x)$ for each element $x \in R_7$. This feat was previously accomplished only by Van Hirtum in 2021 \cite{hirtum}. It is worth noting that the number of elements $x$ in $R_n$ with $\gamma(x) = n!$ for $n > 1$ can be found in the OEIS sequence \seqnum{A220879} (see Table \ref{tab:a220879}). Using the available precalculated sets, we can efficiently determine the $7$th term of the sequence, which is not recorded in the OEIS.

\begin{table}[ht]
\centering
\begin{tabular}[h]{|p{0.25cm}|p{2.2cm}|}
\hline
 $n$ & $\seqnum{A220879} \hspace{0.5mm}$$(n)$\\
\hline
1 & 0 \\
2 & 1 \\
3 & 0 \\
4 & 0 \\
5 & 7 \\
6 & 7281 \\
7 & 468822749 \\
\hline
\end{tabular}
\caption{Inequivalent monotone Boolean functions of $n$ variables with no symmetries.}
\label{tab:a220879}
\end{table}

\begin{table}[ht]
\centering
\begin{tabular}[h]{| c | c |}
\hline
 $k$ & $|\{f\in R_7:\gamma(f) = k\}|$ \\
\hline
1 & 9 \\
7 & 27 \\
21 & 75 \\
30 & 5 \\
35 & 117 \\
42 & 99 \\
70 & 90 \\
84 & 9 \\
105 & 1206 \\
120 & 4 \\
140 & 702 \\
210 & 3255 \\
252 & 114 \\
315 & 2742 \\
360 & 18 \\
420 & 26739 \\
504 & 237 \\
630 & 47242 \\
720 & 4 \\
840 & 75024 \\
1260 & 1024050 \\
1680 & 3128 \\
2520 & 20005503 \\
5040 & 468822749 \\
\hline
\end{tabular}
\caption{Number of $f \in R_7$ with the given $\gamma(f)$.}\label{tab:fr7}
\end{table}

Our program's most critical part, the Boolean function canonization procedure, is based on fast Van Hirtum's approach \cite[Section 5.2.9]{hirtum} and implemented in Rust. Our program is running on a 32-thread machine with Xeon cores.

After the preprocessed data has been loaded into the main memory, the test was performed and the known value of $d_8$ was successfully reproduced in just 16 seconds. However, using this method to check the divisibility of $d_9$ for any value of $m$ is significantly more challenging.

In order to determine which remainders can be computed by our methods, we can use the information in Table \ref{tab:fr7}.
Note that

$$E_{7,m}^c = \sum_{\substack{x \in R_{7} \\ \gamma(x) \bmod m \neq 0}} \gamma(x).$$
The four smallest $E_{7,m}^c$ are $E_{7,7}^c$ with 9999 elements, $E_{7,3}^c$ with 108873 elements, $E_{7,21}^c$ with 118863 elements, and $E_{7,5}^c$ with 154863 elements. Since $d_9$ is already known to be divisible by 3, the next step is to compute remainders of $d_9$ divided by 5 and 7.

\subsection{Remainder of $d_9$ divided by 5}

$$\sum_{x\in R_7\cap E_{7,5}^c}\sum_{y\in E_{7,5}^c} \gamma(x)\cdot G(x,y) = 1404812111893131438640857806,$$
therefore by Theorem \ref{T3}, we have $d_9 \bmod 5 = 1$. We calculated this number in approximately 7 hours.
Moreover, using Theorem \ref{T7} we have $d_9 \equiv 157853570524864492086 \pmod 5$, which confirms that $d_9 \bmod 5 = 1$.

\subsection{Remainder of $d_9$ divided by 7}

$$\sum_{x\in R_7\cap E_{7,7}^c}\sum_{y\in E_{7,7}^c} \gamma(x)\cdot G(x,y) = 299895177645066825375626,$$
therefore by Theorem \ref{T3}, we have $d_9 \bmod 7 = 6$. We calculated this number in approximately half an hour.

  \bigskip
  \hrule
  \bigskip
  
  \noindent 
  2010 \emph{Mathematics Subject Classification}:~Primary 06E30.
  
  \medskip
  
  \noindent 
  \emph{Keywords}:~Dedekind numbers,
  Monotone Boolean functions,
  Ninth Dedekind number,
  Congruence modulo,
  Chinese remainder theorem
  
  \bigskip
  \hrule
  \bigskip
  
  \noindent 
  (Concerned with sequences \seqnum{A000372}, \seqnum{A001206}, \seqnum{A003182} and \seqnum{A220879}.)
  
  \bigskip

\end{document}